\documentclass[submission
]{dmtcs2}

\usepackage[latin1]{inputenc}
\usepackage{subfigure,amsmath,tikz}

\usepackage[round]{natbib}

\newtheorem{theorem}{Theorem}[section]

\newtheorem{proposition}[theorem]{Proposition}
\newtheorem{corollary}[theorem]{Corollary}
\newtheorem{example}[theorem]{Example}

\newtheorem{remark}{Remark}

\newcommand\Z{\mathbb{Z}}

\DeclareMathOperator{\dinv}{dinv}
\DeclareMathOperator{\skips}{skip}
\DeclareMathOperator{\area}{area}
\DeclareMathOperator{\bounce}{bounce}

\DeclareMathOperator{\arm}{arm}
\DeclareMathOperator{\leg}{leg}
\DeclareMathOperator{\rk}{rk}
\DeclareMathOperator{\PF}{PF}
\DeclareMathOperator{\ides}{ides}

\newcommand\catalannumber[6]{
  \begin{tabular}{l}\begin{tikzpicture}[thick,scale=#6, every node/.style={scale=#6}]
  \fill[fill=lime]  (#1) rectangle +(#3,#2);
  \fill[cyan!25]
  (#1)
  \foreach \dir in {#4}{
    \ifnum\dir=0
    -- ++(1,0)
    \else
    -- ++(0,1)
    \fi
  } |- (#1);
  \ifnum #5=1
    \foreach \ivar in {1,...,#3} {
       \foreach \jvar in {1,...,#2} {
         \draw (-.5+\ivar,-.5+\jvar) node {\pgfmathparse{int(-#2*\ivar+(\jvar-1)*#3)}\pgfmathresult};
        }
   }
  \fi
  \draw[help lines] (#1) grid +(#3,#2);
  \draw[dashed] (#1) -- +(#3,#2);
  \coordinate (prev) at (#1);
  \foreach \dir in {#4}{
    \ifnum\dir=0
    \coordinate (dep) at (1,0);
    \else
    \coordinate (dep) at (0,1);
    \fi
    \draw[line width=2pt,-stealth] (prev) -- ++(dep) coordinate (prev);
  };
  \end{tikzpicture}\end{tabular}
}

\newcommand\rankword[2]{
   \newcount\n
   \newcount\m
    \n=\intcalcSub{\intcalcMul{2}{#1}}{3}

   [
   \foreach \ivar in {1,...,\n} {%
     \global\m=0
     \ifnum \intcalcMod{\ivar}{3}=0\else
        \ifnum \intcalcMod{\ivar}{3}=\intcalcMod{#1}{3} \ifnum \ivar<#1
            \foreach\jvar in {#2}{%
               \ifnum \ivar=\jvar
                  \global\m=1
               \fi
             }
             \ifnum \m=1
                \framebox[1.5\width]{\ivar$_2$}
             \else  \ivar$_2$ \fi           
             \ifnum \ivar=\n \else,\fi
         \fi \else 
            \foreach\jvar in {#2}{%
               \ifnum \ivar=\jvar
                  \global\m=1
               \fi
             }
             \ifnum \m=1
                \framebox[1.5\width]{\ivar$_1$}
             \else \ivar$_1$ \fi
             \ifnum \ivar=\n \else,\fi
         \fi
     \fi
   }
   ]
}

\author{Ryan Kaliszewski\thanks{Email: \email{rlk72@drexel.edu}}
  \and Huilan Li\thanks{Email: \email{huilan@math.drexel.edu}}
}
\title[The $(m,n)$-rational $q, t$-Catalan polynomials for $m=3$ and their $q,t$-symmetry]{The $(m,n)$-rational $q, t$-Catalan polynomials for $m=3$ and their $q,t$-symmetry}

\address{Department of Mathematics, Drexel University, 3141 Chestnut Street, Philadelphia, PA 19104, USA}

\begin{document}
\maketitle
\begin{abstract}
We introduce a new statistic, $\skips$, on rational $(3,n)$-Dyck paths  and define a marked rank word for each path when $n$ is not a multiple of 3.  If a triple of valid statistics ($\area,\skips,\dinv$) are given, we have an algorithm to construct the marked rank word corresponding to the triple.  By considering all valid triples we give an explicit formula for the $(m,n)$-rational $q,t$-Catalan polynomials when $m=3$.  Then there is a natural bijection on the triples of statistics $(\area,\skips,\dinv)$ which exchanges the statistics $\area$ and $\dinv$ while fixing the $\skips$. Thus we prove the $q,t$-symmetry of $(m,n)$-rational $q, t$-Catalan polynomials for $m=3$.
\end{abstract}


\section{Introduction}
\label{sec:in}

In the early 1990's Garsia and Haiman introduced an important sum of rational
functions in
$q,t$, the classic $q,t$-Catanlan polynomial $C_n(q,t)$,
which has since been shown to have interpretations in terms of algebraic
geometry and representation theory.
Now it is well-known that the classic $q,t$-Catalan polynomial is
\begin{align} C_n(q,t) & =\sum_{\pi}q^{\dinv(\pi)}t^{\area(\pi)} \\  \label{pt2}  &=\sum_{\pi}q^{\area(\pi)}t^{\bounce(\pi)},\end{align}
where the sums are over all Dyck paths $\pi$ from $(0,0)$ to $(n,n)$ and \eqref{pt2} is due to \cite{Haglund03}.
Though the symmetry problem of the classic $q,t$-Catalan polynomials has been solved through the use of Macdonald polynomials, no bijective proof has been found, where the bijection means an involution on the set of Dyck paths that exchanges  $(\dinv,\area)$ or $(\area,\bounce)$.  For an overview of the classical $q,t$-Catalan polynomials and Dyck paths, see, \cite{Garsia96}, \cite{Garsia02}, \cite{Haglund05}, and \cite{Haglund08}.

Let $m$ and $n$ be coprime. An $(m,n)$-Dyck path is a path in $m\times n$ lattice which proceeds by north and east steps from $(0,0)$ to $(m,n)$ and which always remains weakly above the main diagonal $y=\frac{m}{n}x$. The number of full cells between an $(m,n)$-Dyck path $\Pi$ and the main diagonal is denoted $\area(\Pi)$. The collection of cells above a Dyck path $\Pi$ forms an English Ferrers diagram $\lambda(\Pi).$ For any cell $x\in\lambda(\Pi)$, let $\leg(x)$ and $\arm(x)$ denote number of cells in $\lambda(\Pi)$ which are strictly south or strictly east of $x$, respectively. Then define $\dinv(\Pi)$ is the number of cells above the path $\Pi$ satisfying
\begin{equation} \label{dinveq}\frac{\arm(x)}{\leg(x)+1} < \frac{m}{n} <\frac{\arm(x)+1}{\leg(x)}. \end{equation}
The $(m,n)$-rational $q, t$-Catalan polynomial is 
$$C_{m,n}(q,t)=\sum_{\Pi}q^{\dinv(\Pi)}t^{\area(\Pi)},$$
where the sum is over all $(m,n)$-Dyck paths.  Note that for every $(m,n)$-Dyck path there is a complementary $(n,m)$-Dyck path with identical statistics.  For example, 
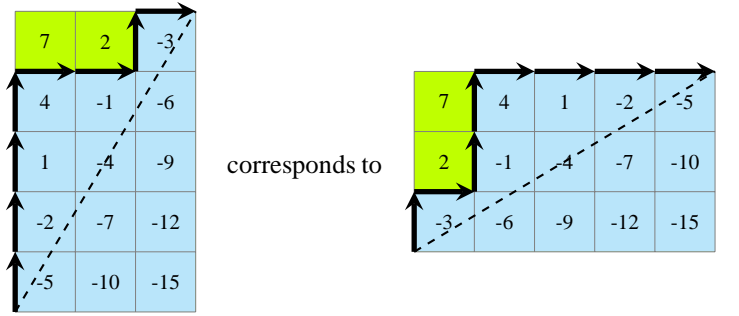
\begin{center}\catalannumber{0,0}{5}{3}{1,1,1,1,0,0,1,0}{1}{0.8} corresponds to \catalannumber{0,0}{3}{5}{1,0,1,1,0,0,0,0}{1}{0.8}.\end{center}
Thus the $(m,n)$-rational $q,t$-Catalan polynomial is symmetric in the parameters $m$ and $n$. 

The classic $q,t$-Catalan polynomials correspond to the cases $m=n+1$. It is an open conjecture that the $(m,n)$-rational $q, t$-Catalan polynomials are symmetric in $q$ and $t$.  \cite{Gorsky14} have proven the $q,t$-symmetry for the case when $m\leq 3$ without giving an explicit formula for the polynomials.  Also, independently, \cite{LLL14} have proven the $q,t$-symmetry when $m\leq 4$.

Using the definition of \cite{Hikita12}, parking functions and their statistics can be extended to the $(m, n)$-rational case. \cite{Hikita12} defined polynomials
\begin{equation} \label{hikitapoly} H_{m,n}(X;q,t)=\sum_{\PF}t^{\area(\PF)}q^{\dinv(\PF)}F_{\ides(\PF)}(X), \end{equation}
where the sum is over all parking functions $\PF$ over all $(m,n)$-Dyck paths and the $F$ are the (Gessel) fundamental quasisymmetric functions indexed by the inverse descent composition of the reading word of the parking functions $\PF$. When $m=2$ or $n=2$, the structure  of Hikita's polynomials  have been completely described by \cite{Leven14}.  Since, for each $(m,n)$-rational Dyck path there is a parking function on that path with the same statistics as the underlying Dyck path, the $(m,n)$-rational $q,t$-Catalan polynomial appears as the coefficient of $F_{(1^n)}$ in the expansion of the Hikita's polynomial $H_{m,n}$.  Thus, describing the structure of the $(m,n)$-rational $q,t$-Catalan polynomials is the first step to uncovering the structure of the Hikita's polynomials.

In this paper, we prove that
\begin{align} C_{3,n}(q,t) & =\sum_{s=0}^{\lfloor n/3 \rfloor}\sum_{\substack{a,d\geq s\\a+s+d+1=n}}q^dt^a \\ &=\sum_{s=0}^{\lfloor n/3 \rfloor}\hspace{1mm}\sum_{a=s}^{n-2s-1}q^{n-a-s-1}t^a \end{align}
and the $q,t$-symmetry, i.e.
\begin{equation} C_{3,n}(q,t)=C_{3,n}(t,q).\end{equation}
For each $(3,n)$-Dyck path, we define a marked rank word corresponding to the path and a new statistic $\skips$ on the path or the marked rank word. If a triple of valid statistics $\area,\skips$ and $\dinv$ are given, we have an algorithm to construct a marked rank word corresponding to the triple. Then there is a natural bijection on the triples of statistics $(\area,\skips,\dinv)$ which switched the statistics $\area$ and $\dinv$. Thus the symmetry in $q,t$ is proved.  \cite{Gorsky14} use a construction similar to our rank words, but the $\skips$ statistic in this paper only arises from the markings we place upon the words.

In Section 2, we introduce the notations, background and define the new statistic $\skips$ on $(3,n)$-Dyck paths. In Section 3, we show that every $(3,n)$-Dyck path is uniquely determined by the triple of statistics $\area,\skips$ and $\dinv$.  We use this triple to construct an explicit formula for $C_{3,n}(q,t)$.  The formula and the uniqueness of the triple allows us to find a unique bijection on $(3,n)$-Dyck paths which switches the $\area$ and $\dinv$ statistics while fixing $\skips$. Therefore, the $q,t$-symmetry of $(m,n)$-rational $q, t$-Catalan polynomials for $m=3$ is proved.

\section{Notation and Background}
\label{sec:first}

Suppose $n$ is a non-negative integer that is not divisible by 3.  Let $L_{3,n}$ be the finite $\Z$-sublattice in the first quadrant with $n$ rows and 3 columns.  We use coordinates to describe the location of each cell in $L_{3,n}$ by first listing the column and then the row.  Rows are labeled from bottom to top and columns are labeled from left to right. So the lower-left corner of cell (1,1) rests on the origin.  We fill each $(a,b)$ cell in the diagram by it's rank: $R(a,b)=-an+3(b-1)$.  If we denote a cell by a variable, $x$, then we will denote the rank of $x$ by $\rk(x)$.

The {\it rank word} of $L_{3,n}$ is a list of the positive ranks of $L_{3,n}$ in increasing order with those ranks in the first column colored with subscript 1's and those in the second column colored with subscript 2's.  The rank word of the sublattice $L_{3,n}$ will be denoted by $\rk(L_{3,n})$.

\begin{example} The rank word $\rk(L_{3,5})$ is $\text{ \emph{\rankword{5}{}}}.$ \end{example}

One can think of the rank word as being composed of two interleaved sublists, one colored with 1's and the other with 2's.  Sublist 1 contains all of the positive integers that are less than $2n$ and equivalent to $2n$ modulo 3, and sublist 2 contains all of the positive integers that are less than $n$ and equivalent to $n$ modulo 3.  Thus there are $\lfloor n/3 \rfloor$ entries with subscript 2 and $\lfloor 2n/3 \rfloor$ entries with subscript 1.  This means that there are $n-1$ entries in $\rk(L_{3,n})$.

The {\it rank word of a $(3,n)$-Dyck path} is 
the rank word of the underlying diagram, $L_{3,n}$, with those entries above the Dyck path are marked.  In this paper we will mark entries by putting a box around them.

\begin{example} If  
\begin{center} $\Pi=$\catalannumber{0,0}{5}{3}{1,1,1,1,0,0,1,0}{1}{0.9},\end{center}
 then 
\[\rk(\Pi)=\text{\emph{\rankword{5}{2,7}}}.\]
\end{example}

If $\Pi$ is a $(3,n)$-Dyck path, define its {\it skip} as a collection of adjacent unboxed entries in the rank word with one or more boxed entries to the right and left.  In the above example, the 4 is a skip because the 2 and 7 are both in boxes.  The 1 is not a skip because there are no boxed entries to the left.  Define the {\it skip statistic} of $\Pi$ to be the total number of skips, denoted by $\skips(\Pi)$.

\begin{example} Let
\begin{center}
$\Pi_1=$
\catalannumber{0,0}{8}{3}{1,1,1,1,1,1,0,0,1,1,0}{1}{0.8}
and $\Pi_2=$
\catalannumber{0,0}{8}{3}{1,1,1,1,1,1,1,0,0,1,0}{1}{0.8}.
\end{center}
This gives
\[\rk(\Pi_1)=\text{\emph{\rankword{8}{2,5,10,13}}}\]
\[\rk(\Pi_2)=\text{\emph{\rankword{8}{5,13}}}.\]
Therefore $skip(\Pi_1)$ is 2 because the entries 4 and 7 are each skips.  However, $skip(\Pi_2)$ is only 1 because the 7 and 10 entries comprise a single skip.  It is useful to confirm that for the Dyck paths above, $\area(\Pi_1)=3$, $\area(\Pi_2)=5$, $\dinv(\Pi_1)=2$, and $\dinv(\Pi_2)=1$.
\end{example}

We will turn our discussion to the $\dinv$ statistic for the rest of this section.  Note that the ranks in $\rk(L_{3,n})$ only lie in the first and second columns of $L_{3,n}$.

\begin{proposition}
Any cell in the second column that is above a $(3,n)$-Dyck path satisfies the inequality \eqref{dinveq}, and thus contributes to the $\dinv$.
\end{proposition}
\begin{proof}
As stated earlier, there are at most $\lfloor n/3 \rfloor$ cells above the path that lie in the second column.  Let $x$ be one such cell.  This means that $\arm(x)=0$ and $\leg(x)<n/3$.  Since $\arm(x)=0$ the left inequality in \eqref{dinveq} certainly holds, so we just need to show that $3\cdot\leg(x)<n$.  But this is exactly what we just claimed.
\end{proof}

\begin{proposition} \label{classifyprop}
Let $\Pi$ be a $(3,n)$-Dyck path.  Any cell $x$ in the first column and above $\Pi$ that does \emph{not} satisfy the inequality \eqref{dinveq} must satisfy one of the following:

\begin{equation} \label{cond1} \arm(x)=1 \text{ and } \leg(x)<\frac{n}{3}-1, \end{equation}
\begin{equation} \label{cond2} \arm(x)=0 \text{ and }\leg(x)>\frac{n}{3}. \end{equation}

\end{proposition}

\begin{proof}
Let $x$ be a cell in the first column above the path $\Pi$ that does not satisfy the inequality  \eqref{dinveq}. Then $\arm(x)=1$ or $\arm(x)=0$.  

Suppose $\arm(x)=1$.  Then the inequality \eqref{dinveq} becomes
\[ \frac{1}{\leg(x)+1} < \frac{3}{n} <\frac{2}{\leg(x)}. \]
From the right inequality, then we have $\leg(x)<\frac{2n}{3}$.  As we previously discussed, there are never more than $\lfloor 2n/3 \rfloor$ positive entries above the path in the first column.  So the right inequality is always satisfied.  Therefore, the left inequality does not hold, which is equivalent to $\leg(x)<\frac{n}{3}-1$, which is line~\eqref{cond1}.

Suppose $\arm(x)=0$.  Then the inequality \eqref{dinveq} becomes
\[ \frac{0}{\leg(x)+1} < \frac{3}{n} <\frac{1}{\leg(x)}, \]
so the left inequality is immediately satisfied.  Therefore, the right inequality does not hold, which is equivalent to $\leg(x)>\frac{n}{3}$, which is line~\eqref{cond2}.

\end{proof}

This proposition tells us that the only cells above the Dyck path that do not contribute to the $\dinv$ are those without a cell immediately to the right (above the path) with many cells below or those with a cell immediately to the right (above the path) with few cells below.  Thus we can have one case or the other, but not both.

Now consider:
\begin{proposition}
If $\Pi$ is a $(3,n)$-Dyck path then 
\[ \area(\Pi)+\skips(\Pi)+\dinv(\Pi)=|rk(L_{3,n})|=n-1.\]
\end{proposition}

\begin{proof}
It is sufficient to show that for every cell above the path that does not contribute to the $\dinv$ of $\Pi$ there is a unique corresponding skip in the rank word of $\Pi$.  Let $x$ be a such cell that does not satisfy the inequality \eqref{dinveq}.  Then by Proposition \ref{classifyprop}, $x$ satisfies one of two cases:

\emph{Case 1: $\arm(x)=1$ and $\leg(x)<\frac{n}{3}-1$}.  In this case it is certainly true that $\leg(x)<\lfloor n/3 \rfloor$.  Let $\underbar{\emph{x}}$ be the cell that is directly $\lfloor n/3 \rfloor$ cells below $x$.  We know that $\underbar{\emph{x}}$ must lie below the path.  We will show that for each $x_i$ that does not contribute to the $\dinv$, the rank of $\underbar{\emph{x}}_i$ will lie in a distinct skip.

Suppose $x$ is a cell with $\arm(x)=1$ that does not contribute to the $\dinv$. From $\arm(x)=1$, $\rk(x)>n$. So $\rk(\underbar{\emph{x}})=\rk(x)-3\cdot\lfloor n/3 \rfloor >0$. Let $\hat{x}$ be the cell immediately to the right of $x$. So $\rk(\hat{x})=\rk(x)-n$. From $n>3\cdot\lfloor n/3\rfloor$, $$\rk(\hat{x})<\rk(\underbar{\emph{x}})<\rk(x).$$ Since $x$ and $\hat{x}$ are both above the path, $\rk(\underbar{\emph{x}})$ must be in a skip.

Suppose that $x$ and $y$ are two cells with  $\arm(x)=\arm(y)=1$  that do not contribute to the $\dinv$. Assume that $\rk(x)>\rk(y)$. So $\rk(x)-3\geq\rk(y)$ and $\rk(y)>\rk(\underbar{\emph{x}})$. From $3\cdot\lfloor n/3\rfloor +3>n$, $$\rk(\underbar{\emph{y}})=\rk(y)-3\cdot\lfloor n/3 \rfloor\leq\rk(x)-3\cdot\lfloor n/3 \rfloor-3<\rk(x)-n=\rk(\hat{x}).$$ So
$$\rk(\hat{y})<\rk(\underbar{\emph{y}})<\rk(\hat{x})<\rk(\underbar{\emph{x}})<\rk(x),$$
i.e., $\rk(\underbar{\emph{x}})$ and $\rk(\underbar{\emph{y}})$ must be in distinct skips.

\emph{Case 2: $\arm(x)=0$ and $\leg(x)>\frac{n}{3}$}.  From $arm(x)=0$, $\hat{x}$ must lie below the path.  We will show that for each $x_i$ that does not contribute to the $\dinv$, the rank of $\hat{x_i}$ will lie in a distinct skip.

Let $\bar{x}$ be the cell $\lceil n/3 \rceil$ below cell $x$. Since $\leg(x)>\frac{n}{3}$ then cell $\bar{x}$ is above the path. From $3\cdot\lceil n/3 \rceil>n$, $\rk(\bar{x})=\rk(x)-3\cdot\lceil n/3 \rceil <\rk(x)-n=\rk(\hat{x})$. So we have 
$$\rk(\bar{x})<\rk(\hat{x})<\rk(x).$$ 
Since $x$ and $\bar{x}$ are both above the path, $\rk(\hat{x})$ must be in a skip.

Suppose that $x$ and $y$ are two cells with  $\arm(x)=\arm(y)=0$  that do not contribute to the $\dinv$.
Assume that $\rk(x)>\rk(y)$. So $\rk(x)-3\geq\rk(y)$ and $\rk(y)>\rk(\bar{x})$. From $3\cdot\lceil n/3 \rceil>n+3$, $$\rk(\hat{y})=\rk(y)-n\leq\rk(x)-n-3<\rk(x)-3\cdot\lceil n/3 \rceil  =\rk(\bar{x}).$$So
$$\rk(\bar{y})<\rk(\hat{y})<\rk(\bar{x})<\rk(\hat{x})<\rk(x),$$
i.e., $\rk(\hat{x})$ and $\rk(\hat{y})$ must be in distinct skips.
\end{proof}

\section{Uniqueness of rank words}
In this section we will show that every $(3,n)$-Dyck path is uniquely determined by the $\dinv$, $\area$, and $\skips$ statistics.  That is if $\Pi_1$ and $\Pi_2$ are Dyck paths such that
\[ \area(\Pi_1)=\area(\Pi_2), \]
\[ \skips(\Pi_1)=\skips(\Pi_2), \]
\[ \dinv(\Pi_1)=\dinv(\Pi_2), \]
then $\Pi_1=\Pi_2$.  We know that $\area(\Pi_1)+\skips(\Pi_1)+\dinv(\Pi_1)=|\rk(L_{3,n})|=n-1$.  Therefore it is immediate that both paths must be inscribed in the same lattice.  So let us consider an algorithm that generates a rank word, given a Dyck path (later we will prove that it is the rank word associated to the Dyck path):

\vspace{3mm}
\noindent\textbf{Path rank word construction algorithm} \\
\textbf{Data:} Three non-negative integers, $a, s,$ and $d$ such that $a=\area(\Pi), s=\skips(\Pi),$ and $d=\dinv(\Pi)$\\ \hbox{\hskip 9mm} for some $(3,n)$-Dyck path. \\
\textbf{Result:} The rank word of $\Pi$, $\rk(\Pi)$.\\
Let $n=a+s+d+1$ and set $r=\rk(L_{3,n})$, the rank word of the diagram.\\
Put boxes around the rightmost $d$ entries of $r$. \\
\textbf{For} i = 1 to s \\
\hbox{\hskip 5mm} Moving left, skip all successive entries of the same color.\\
\hbox{\hskip 5mm} Box the rightmost entry that hasn't been skipped or boxed.\\
\textbf{end}

\begin{example} Consider the $(3,8)$-Dyck paths:

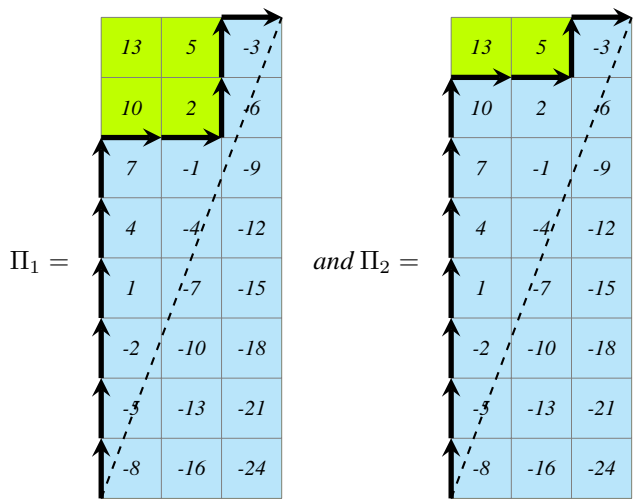
\begin{center}
$\Pi_1=$
\catalannumber{0,0}{8}{3}{1,1,1,1,1,1,0,0,1,1,0}{1}{0.8}
and $\Pi_2=$
\catalannumber{0,0}{8}{3}{1,1,1,1,1,1,1,0,0,1,0}{1}{0.8}.
\end{center}
We can compute that $\dinv(\Pi_1)=2$ and $\skips(\Pi_1)=2$.  If we follow the algorithm we have:\\
\begin{tabular}{ccp{6.6cm}}
Initial: & \emph{\rankword{8}{10,13}} & because $d=2$. \\
$i=1$: & \emph{\rankword{8}{5,10,13}} & because we skipped the 7 which has color 1 and boxed the rightmost entry that hasn't been skipped or boxed, the 5. \\
$i=2$: & \emph{\rankword{8}{2,5,10,13}} & because we skip the 4, and boxed the rightmost entry that hasn't been skipped or boxed, the 2.
\end{tabular}\\
Since $s=2$ we only perform two iterations of the loop, so we are done. 

Similarly for $\Pi_2$ we have:\\
\begin{tabular}{ccp{6.9cm}}
Initial: & \emph{\rankword{8}{13}} & because $d=1$. \\
$i=1$: & \emph{\rankword{8}{5,13}} & we skip the 10 and the 7 because they are adjacent and have the same color.  Then we box the rightmost entry which hasn't been boxed or skipped, which is the 5.
\end{tabular}\\
We have finished the algorithm at this point because $s=1$, so we only performed one iteration of the loop.
\end{example}

This leads us to the following proposition:
\begin{proposition} If $\Pi$ is a $(3,n)$-Dyck path and we use the path rank word construction algorithm to generate a word $\omega(\Pi)$, then $\omega(\Pi)=\rk(\Pi)$.
\end{proposition}
\begin{proof}
To construct a $(3,n)$-Dyck path it is equivalent to choose a number of cells in the first column as well as a number of cells in the second column (subject to constraints)  to be above the path.  This is essentially choosing a number of entries colored 1 as well as a number of entries colored 2 in the $\rk(L_{3,n})$ to box .  Moreover, all of the boxed entries of a specific color must be the rightmost entries.  For instance, in the previous example, when we constructed the rank word of $\Pi_1$ we boxed the two rightmost entries colored with 1 and the two rightmost entries colored with 2.  Suppose that there are $k$ cells above the path in the first column and $\ell$ cells above the path in the second column.

\emph{Case 1:  Suppose that $k<n/3$.}  We will show that $k=\dinv(\Pi)$ and $\ell=\skips(\Pi)$.  From Proposition \ref{classifyprop}, each cell in column one that has $\arm$ equal to zero will contribute to the $\dinv$ of $\Pi$ because the $\leg$ is certainly less than $n/3$.  Every cell in column one that has $\arm$ equal to one does not contribute to the $\dinv$.  

Suppose $\ell>0$. Consider the upper left cell, call it $x$. We have $\arm(x)=1$ and  $\leg(x)=k-1 <\frac{n}{3}-1,$ i.e., $x$ does not contribute to $\dinv$. The other cells in column one with arm one have shorter legs. They will not contribute to $\dinv$ either.

In $\rk(L_{3,n})$ the rightmost $\dinv(\Pi)=k$ entries will be boxed, and they are all color 1.  Then the remaining rightmost 1-colored entries will be skipped (due to skipping) and the rightmost $\skips(\Pi)=\ell$ 2-colored entries will be boxed.  This concludes the proof of this case.

\emph{Case 2:  Suppose that $k>n/3$ and $k-\ell >\frac{n}{3}+1$.}  From Proposition \ref{classifyprop},  $\skips(\Pi)$ is equal to the number of cells with $\arm$ zero and whose $\leg$ is greater than $n/3$.  There are $k-\ell-\lceil n/3 \rceil$ such cells so $\skips(\Pi)=k-\ell-\lceil n/3\rceil$ and $\dinv(\Pi)=2\ell+\lceil n/3 \rceil$.

In $\rk(L_{3,n})$ there are $\lceil n/3\rceil$ entries with color 1 on the right and then the entries alternate in color, beginning with 2.  Thus we will box $\lceil n/3 \rceil$ color 1 entries, then $\ell$ entries with color 2 and $\ell$ entries with color 1.  Since the last entry we boxed was 1, we will begin skipping color 2 entries and box $k-\ell-\lceil n/3\rceil$ color 1 entries.  Therefore the rightmost $\ell$ color 2 entries are boxed and the rightmost $\lceil n/3\rceil+\ell+\left(k-\ell-\lceil n/3\rceil\right)=k$ color 1 entries are boxed.  This completes the proof for this case.

\emph{Case 3:  Suppose that $k>n/3$ and $k-\ell<\frac{n}{3}+1$.}  From Proposition \ref{classifyprop},  $\skips(\Pi)$ is equal to the number of cells with $\arm$ equal to one and whose $\leg$ is less than $\frac{n}{3}-1$.  There are $\ell-k+\lfloor n/3\rfloor$ of these
.  Thus $\skips(\Pi)=\ell-k+\lfloor n/3\rfloor$ and $\dinv(\Pi)=2k-\lfloor n/3\rfloor$.

In $\rk(L_{3,n})$ there are $\lceil n/3\rceil$ color 1 entries on the right and then the entries alternate in color, beginning with 2.  Thus we will box $\lceil n/3\rceil$ color 1 entries.  Then we will box $2k-\lfloor n/3\rfloor - \lceil n/3\rceil=2k-2\cdot\lfloor n/3\rfloor -1$ entries: $k-\lfloor n/3 \rfloor$ color 2 entires and $k-\lfloor n/3 \rfloor -1$ color 1 ones. Then we will box $\ell-k+\lfloor n/3\rfloor$ color 2 entries at the same time skip $\ell-k+\lfloor n/3\rfloor$ color 1 entries. So we have boxed $\lceil n/3 \rceil + k-\lfloor n/3\rfloor -1=k$ entries with color 1 and $k-\lfloor n/3\rfloor + \ell-k+\lfloor n/3 \rfloor=\ell$ entries with color 2.  This completes the proof.
\end{proof}

\begin{remark} Any rank word with the rightmost $k$ color 1 entries boxed, the rightmost $\ell$ color 2 entries boxed, and $k\geq \ell$ corresponds uniquely to a $(3, n)$-Dyck path.  
\end{remark}

In order to satisfy these conditions, the following inequalities must hold for any $(3, n)$-Dyck path $\Pi$:
\begin{enumerate}
\item $0\leq\skips(\Pi)< n/3$,
\item $\skips(\Pi)\leq \dinv(\Pi) \leq n-1-2\cdot\skips(\Pi)$,
\item $\skips(\Pi)\leq \area(\Pi)\leq n-1-2\cdot\skips(\Pi)$.
\end{enumerate}

The first inequality is clear because either the rank word skips only color 2 entries (of which there are $\lfloor n/3\rfloor$) or skips only color 1 entries.  There are $\lfloor2n/3\rfloor$ color 1 entries, but the first $\lceil n/3 \rceil$ entries only contribute at most one skip. We know $\lfloor2n/3\rfloor-\lceil n/3 \rceil \leq\lfloor n/3\rfloor$. When $\lfloor2n/3\rfloor-\lceil n/3 \rceil <\lfloor n/3\rfloor$, the first entry of the rank work has color 2. So, totally, there are at most $\lfloor n/3\rfloor$ entries with color 1 that contribute skips.   When $\lfloor2n/3\rfloor-\lceil n/3 \rceil =\lfloor n/3\rfloor$, the first entry of the rank work is color 1, which can not be a skip. Again, totally, there are at most $\lfloor n/3\rfloor$ entries with color 1 that contribute skips. 

For the second inequality, if $\dinv(\Pi)\geq\lceil n/3 \rceil$, from $\skips(\Pi)< n/3$ we have $\skips(\Pi)\leq \dinv(\Pi)$.  If $\dinv(\Pi)<\lceil n/3 \rceil$, then only color 1 entries are skipped. The number of boxed color 1 entries is $\dinv(\Pi)$ and the number of boxed color 2 entries is $\skips(\Pi)$. Recall that there must be at least as many color 1 ranks boxed as color 2 ranks boxed in the rank word of a path. So $\skips(\Pi)\leq \dinv(\Pi)$.

The right part of the second inequality equivalent to $\skips(\Pi)\leq\area(\Pi)$ (recall that $\area(\Pi)+\skips(\Pi)+\dinv(\Pi)=n-1$).  It follows from the fact that the area of a path is the number of unboxed entries in the rank word and a skip must always skip at least one entry.  From this argument one can also see that the second and third inequalities are actually equivalent.

The definition of $\omega$, above, does not actually require a path.  If we have three non-negative integers that satisfy the above inequalities we can generate a rank word, call it $\omega(\area,\skips,\dinv)$.

\begin{corollary} \label{abovecor} Given three non-negative integers $a,s,d$ that satisfy
\begin{enumerate}
\item $s\leq a$,
\item $s\leq d$,
\item $a+s+d+1$ is not divisible by $3$,
\end{enumerate}
then there is a unique $(3,a+s+d+1)$-Dyck path $\Pi$ with $a=\area(\Pi), s=\skips(\Pi),$ and $d=\dinv(\Pi)$.
\end{corollary}

\begin{proof}  Set $n=a+s+d+1$.  We only need to check that $\omega(a,s,d)$ has enough entries to ensure that $s$ skips can be performed.  If $d\geq\lceil n/3 \rceil$ then each skip corresponds to two entries, one skipped and one not skipped.  So we need $2s+d\leq s+a+d$, which is true.  If $d<\lceil n/3 \rceil$ then the first skip and the rightmost $d$ entries together encompass $\lceil n/3 \rceil+1$ of the total entries.  The remaining skips then correspond to two entries. Again, we need $2(s-1)+\lceil n/3 \rceil+1\leq s+a+d$, which is equivalent to
\begin{equation}\label{eqstar} s-1+\lceil n/3 \rceil \leq a+d.\end{equation}
Recall that area is the unboxed entries.  Since there were $\lceil n/3\rceil -d$ skipped entries at the first skip and $s-1$ skipped entries in the middle,  $a\geq \lceil n/3\rceil -d + s-1$, which is exactly inequality \eqref{eqstar}.
\end{proof}

\begin{remark} Note that if $s\leq a,d$ and $n=a+s+d+1$ then immediately $s<n/3$.
\end{remark}

This leads us to the following theorem:
\begin{theorem}
\begin{align} C_{3,n}(q,t) & =\sum_{s=0}^{\lfloor n/3 \rfloor}\sum_{\substack{a,d\geq s\\a+s+d+1=n}}q^dt^a \\ &=\sum_{s=0}^{\lfloor n/3 \rfloor}\hspace{1mm}\sum_{a=s}^{n-2s-1}q^{n-a-s-1}t^a \end{align}
\end{theorem}

\begin{theorem}
There is a unique bijection on the set of $(3,n)$-Dyck paths that exchanges $\area$ and $\dinv$ while fixing $\skips$.
\end{theorem}

\begin{proof}  This follows from the fact that the $\area, \skips$, and $\dinv$ statistics uniquely define a path and for every $(a,s,d)$ triple corresponding to a path, $(d,s,a)$ also corresponds to a path.
\end{proof}

\begin{corollary}
The $(3,n)$-rational $q,t$-Catalan polynomials are symmetric in $q$ and $t$, i.e., 
\begin{equation} C_{3,n}(q,t)=C_{3,n}(t,q). \end{equation}
\end{corollary}

This theorem leads us in two natural directions for further research.  First, generalize the $\skips$ statistic  to arbitrary $(m,n)$-Dyck paths, which would give insight into finding an involution that exchanges $\area$ and $\dinv$.  Second, completely describe Hikita's polynomial $H_{3,n}$ \eqref{hikitapoly} since we know the coefficient of $F_{(1^n)}$.

\acknowledgements
\label{sec:ack} We would like to thank Adriano Garsia and Jim Haglund for the motivation to work
on this problem. We wish to thank Emily Leven for inspiring conversations on this topic.  We are also grateful to the anonymous reviewers for providing useful comments and papers to compare to our work.

\bibliographystyle{abbrvnat}

\bibliography{catalanbib}
\label{sec:biblio}

\end{document}